\documentclass[10pt]{amsart}

\usepackage{amsfonts,amssymb,amsmath,amsthm,amscd,enumerate}
\usepackage{latexsym}
\usepackage{euscript}
\usepackage[english]{babel}
\usepackage[latin1]{inputenc}

\newtheorem{theorem}{Theorem}[section]
\newtheorem{definition}[theorem]{Definition}
\newtheorem{lemma}[theorem]{Lemma}
\newtheorem{corollary}[theorem]{Corollary}
\newtheorem{proposition}[theorem]{Proposition}
\newtheorem{Remark}[theorem]{Remark}
\newtheorem{example}{Example}

\title[$sl_{2}$ over a finite field]{On $\mathbb{Z}_{2}$-graded polynomial identities of $sl_{2}(F)$ over a finite field}
\author{Luís Felipe Gonçalves Fonseca}
\address{Departamento de Matemática, Universidade Federal
de Viçosa - Campus Florestal, Rodovia LMG 818, km 06, Florestal, MG,
Brazil} \email{luisfelipe@ufv.br}

\begin{document}
\maketitle

\begin{abstract}

Let $F$ be a finite field of $char F > 3$ and $sl_{2}(F)$ be the Lie
algebra of traceless $2\times 2$ matrices over $F$. In this paper,
we find a basis for the $\mathbb{Z}_{2}$-graded identities of
$sl_{2}(F)$.

\vspace{0.5cm}


Mathematics Subject Classification 2010: Graded identities; Lie
algebras; finite basis identities.

\end{abstract}

\smallskip




\section{Introduction}

The well-known Ado-Iwasawa' theorem posits that any
finite-dimensional Lie algebra over an arbitrary field has a
faithful finite-dimensional representation. Briefly, any finite
dimensional Lie algebra can be viewed as a subalgebra of a Lie
algebra of square matrices under the commutator brackets. Thus, the
study of Lie algebras of matrices is of considerable interest.

A task in PI-theory is to describe the identities of $sl_{2}(F)$,
the Lie algebra of traceless $2\times 2$ matrices over a field $F$
of $char F \neq 2$. The first breakthrough in this area was made by
Razmyslov \cite{Razmyslov}, who described a basis for the identities
of $sl_{2}(F)$ when $char F = 0$. Vasilovsky \cite{Vasilovsky} found
a single identity for the identities of $sl_{2}(F)$ when $F$ is an
infinite field of $char F \neq 2$, and Semenov \cite{Semenov}
described a basis (with two identities) for the identities of
$sl_{2}(F)$ when $F$ is a finite field of $char F
> 3$.

The Lie algebra $sl_{2}(F)$ can be naturally graded by
$\mathbb{Z}_{2}$ as follows: $sl_{2}(F) =
\newline (sl_{2}(F))_{0}\oplus (sl_{2}(F))_{1}$ where $(sl_{2}(F))_{0},
(sl_{2}(F))_{1}$ contain diagonal and off-diagonal matrices
respectively. A recent development in PI-theory is the description
of the graded identities of $sl_{2}(F)$. Using invariant theory
techniques, Koshlukov \cite{Plamen1} described the
$\mathbb{Z}_{2}$-graded identities for $sl_{2}(F)$ when $F$ is an
infinite field of $char F \neq 2$. Several further papers on graded
identities of $sl_{2}(F)$ have appeared in recent years (cf. e.g.,
\cite{Giambruno} and \cite{Giambruno2}). In these studies, the
ground field is of characteristic zero.

To date, no basis has been found for the $\mathbb{Z}_{2}$-graded
identities of $sl_{2}(F)$ when $F$ is a finite field.

In this paper we give a basis for the graded identities $sl_{2}(F)$
when $F$ is a finite field of $char F
> 3$.

\section{Preliminaries}


Let $F$ be a fixed finite field of $char F > 3$ and size $|F| = q$,
let $\mathbb{N}_{0} = \{1,2,\ldots,\newline, n,\ldots\}$, let $G =
(\mathbb{Z}_{2},+)$, and let $L$ be a Lie algebra over $F$. In this
study (unless otherwise stated), all vector spaces and Lie algebras
are considered over $F$. The $+^{\cdot},\oplus,
span_{F}\{a_{1},\ldots,a_{n}\}, \langle a_{1},\ldots,a_{n} \rangle,
(a_{1},\ldots,a_{n} \in L)$ signs denote the direct sum of Lie
algebras, the direct sum of vector spaces, the vector space
generated by $a_{1},\ldots,a_{n}$, and the ideal generated by
$a_{1},\ldots,a_{n}$ respectively, while an associative product is
represented by a dot: $``.''$. The commutator $([,])$ denotes the
multiplication operation of a Lie algebra. We assume that all
commutators are left-normed, i.e., $[x_{1},x_{2},\ldots,x_{n}]:=
[[x_{1},x_{2},\ldots,x_{n-1}],x_{n}] \ \ n \geq 3 $. We use the
convention $[x_{1},x_{2}^{k}] = [x_{1},x_{2},\ldots,x_{2}]$, where
$x_{2}$ appears $k$ times in the expanded commutator.


We denote by $gl_{2}(F)$ the Lie algebra of $2 \times 2$ matrices
over $F$. Let $sl_{2}(F)$ denote the Lie algebra of traceless
$2\times 2$ matrices over $F$. Here, $e_{ij} \subset gl_{2}(F)$
denotes the unitary matrix unit whose elements are $1$ in the
positions $(ij)$ and $0$ otherwise.


The basic concepts of Lie algebra adopted in this study can be found
in Chapters 1 and 2 of \cite{Humphreys}. We denote the center of $L$
by $Z(L)$. If $x \in L$, we denote by $ad x$ the linear map with the
function rule: $y \mapsto [x,y]$. $L$ is said to be metabelian if it
is solvable of class at most $2$. As is known, if $L$ ($L$ over a
finite field of $char F
> 3$) is a three-dimensional simple Lie algebra, then $L \cong
sl_{2}(F)$. $L$ is regarded as a Lie $A$-algebra if all of its
nilpotent subalgebras are abelian.


A Lie algebra $L$ is said to be $G$-graded (a graded Lie algebra or
graded by $G$) when there exist subspaces $\{L_{g}\}_{g \in G}
\subset L$ such that $L = \bigoplus_{g \in G}L_{g}$, and
$[L_{g},L_{h}] \subset L_{g+h}$ for any $g,h \in G$. $G$-graded
associative algebras are defined in the same way. In that context,
$\{L_{g}\}_{g \in G}$ is said to be a grading for $L$. An element
$a$ is called homogeneous when $a \in \bigcup_{g \in G} L_{g}$. We
say that $a \neq 0$ is a homogeneous element of $G$-degree $g$ when
$a \in L_{g}$. A $G$-graded homomorphism of two $G$-graded Lie
algebras $L_{1}$ and $L_{2}$ is a homomorphism $\phi: L_{1}
\rightarrow L_{2}$ such that $\phi({L_{1}}_{g}) \subset {L_{2}}_{g}$
for all $g \in G$. Two gradings on $L$ $\{L_{g}\}_{g \in G}$ and
$\{L'_{g}\}_{g \in G}$ on $L$ are called isomorphic when there
exists a $G$-graded isomorphism $\phi: L \rightarrow L$ such that
$\phi(L_{g}) = L'_{g}$ for all $g \in G$. An ideal $I \subset L$ is
graded when $I = \bigoplus_{g \in G}(I\cap L_{g})$ (we define graded
Lie subalgebras similarly). Likewise, if $I$ is a graded ideal of
$L$, $C_{L}(I) = \{a \in L | [a,I] = \{0\}\}$ is also a graded ideal
of $L$. Furthermore, $Z(L), L^{n}$ (the $n$-th term of descending
central series), and $L^{(n)}$ (the $n$-th term of derived series)
are graded ideals of $L$. We use the convention that $L^{(1)} =
[L,L]$ and $L^{1} = L$.


Let $L$ be a finite-dimensional Lie algebra. Denote by $Nil(L)$ the
greatest nilpotent ideal of $L$ and by $Rad(L)$ the greatest
solvable ideal of $L$. Clearly, $Nil(L)$ is the unique maximal
abelian ideal of $L$ when $L$ is a Lie $A$-algebra. Furthermore,
every subalgebra and every factor algebra of $L$ is a Lie
$A$-algebra when $L$ is also a Lie $A$-algebra (see Lemma 2.1 in
\cite{Towers} and Lemma 1 in \cite{Semenov2}).


The next theorem is a structural result on solvable Lie
$A$-algebras.

\begin{theorem}[Towers, Theorem 3.5, \cite{Towers}]\label{david}
Let $L$ be a (finite-dimensional) solvable Lie $A$-algebra (over an
arbitrary field $F$) of derived length $n + 1$ with nilradical
$Nil(L)$. Moreover, let $K$ be an ideal of $L$ and $B$ a minimal
ideal of $L$. Then we have the following:
\begin{description}
\item $K = (K\cap A_{n}) \oplus (K\cap A_{n-1}) \oplus \ldots \oplus (K\cap
A_{0})$;
\item $Nil(L) = A_{n} +^{\cdot} (A_{n-1}\cap Nil(L)) +^{\cdot} \ldots +^{\cdot} (A_{0}\cap
Nil(L))$;
\item $Z(L^{(i)}) = Nil(L)\cap A_{i}$ for each $0 \leq i \leq n$;
\item $B \subseteqq Nil(L)\cap A_{i}$ for some $0 \leq i \leq n$.
\end{description}
$A_{n} = L^{(n)}$, $A_{n-1},\ldots,A_{0}$ are abelian subalgebras of
$L$ defined in the proof of Corollary 3.2 in \cite{Towers}.
\end{theorem}
\begin{Remark}
Assuming Theorem \ref{david}, we can prove that, if $L =
\bigoplus_{g \in G}L_{g}$ is a (finite-dimensional) solvable graded
Lie $A$-algebra (over an arbitrary field $F$) of derived length $n +
1$ with nilradical $Nil(L)$ , then $Nil(L)$ is a graded ideal.
Moreover, if $L$ is finite-dimensional metabelian Lie $A$-algebra
(over an arbitrary field), then $Nil(L) = [L,L] +^{\cdot} Z(L)$.
\end{Remark}


A finite-dimensional Lie algebra $L$ is called semisimple if $Rad(L)
= \{0\}$. Recall that $L$ (finite-dimensional and non solvable) has
a Levi decomposition when there exist a semisimple subalgebra $S
\neq \{0\}$ (termed a Levi subalgebra) such that $L$ is a semidirect
product of $S$ and $Rad(L)$. We now present a result.

\begin{proposition}[Premet and Semenov, Proposition 2, adapted,
\cite{Semenov2}]\label{premet} Let $L$ be a finite-dimensional Lie
$A$-algebra over a finite field $F$ of $char F > 3$. Then,
\begin{description}
\item $[L,L]\cap Z(L) = \{0\}$.
\item $L$ has a Levi decomposition. Moreover, each Levi subalgebra
$S$ is represented as a direct sum of $F$-simple ideals in $S$, each
one of which splits over some finite extension of the ground field
into a direct sum of the ideals isomorphic to $sl_{2}(F)$.
\end{description}
\end{proposition}


A Lie algebra $L$ is said to be $G$-simple if $[L,L] \neq \{0\}$,
and $L$ does not have any proper non-trivial graded ideals.

By mimicking the arguments of Zaicev et al. in \cite{Zaicev} (Lemma
2.1; Section 3; Proposition 3.1, items i and ii), we have the
following.

\begin{proposition}\label{premet3}
Let $L$ be a finite dimensional graded Lie algebra. The ideal
$Rad(L)$ is a graded ideal. If $L$ is $G$-simple, then $L$ is a
direct sum of simple Lie algebras. If $L$ is direct sum of simple
Lie algebras, then $L$ is a direct sum of $G$-simple Lie algebras.
\end{proposition}

\section{Graded identities and varieties of graded Lie algebras}


Let $X = \{X_{g} = \{x_{1}^{g},\ldots,x_{n}^{g},\ldots\} | g \in
G\}$ be a class of pairwise-disjoint enumerable sets, where $X_{g}$
denotes the variables of $G$-degree $g$. Let $F\langle X \rangle$ be
the free associative unital algebra and let $L(X)$ be the Lie
subalgebra of $F\langle X \rangle$ generated by $X$. $L(X)$ is known
to be isomorphic to the free Lie algebra with a set of free
generators $X$. The algebras $L(X)$ and $F\langle X \rangle$ have
natural $G$-grading. A graded ideal $I \subset L(X)$ invariant under
all graded endomorphisms is called a graded verbal ideal. Let $S
\subset L(X)$ be a non empty set. The graded verbal ideal generated
by $S$, $\langle S \rangle_{T}$, is defined as the intersection of
all verbal ideals containing $S$. A polynomial $f \in L(X)$ is
called a consequence of $g \in L(X)$ when $f \in \langle g
\rangle_{T}$, and it is called a graded polynomial identity for a
graded Lie algebra $L$ if $f$ vanishes on $L$ whenever the variables
from $X_{g}$ are substituted by elements of $L_{g}$ for all $g \in
G$. We denote by $Id_{G}(L)$ the set of all graded identities of
$L$. The variety determined by $S \subset L(X)$ is denoted by
$\mathcal{V}(S) = \{A \ \mbox{is a} \ \mbox{G-graded Lie algebra} |
Id_{G}(A) \supset \langle S \rangle_{T}\}$. The variety generated by
a graded Lie algebra $L$ is denoted by $var_{G}(L) = \{A \ \mbox{is
a} \ \ \mbox{G-graded Lie algebra} | Id_{G}(L) \subset Id_{G}(A)\}$.
We say that a class of graded Lie algebras $\{L_{i}\}_{i \in
\Gamma}$, where $\Gamma$ is an index set, generates $\mathcal{V}(S)$
when $\langle S \rangle_{T} = \bigcap_{i \in \Gamma} Id_{G}(L_{i})$.

We denote by $Id(L)$ the set of all ordinary identities of a Lie
algebra $L$, and by $var(L)$ the variety generated by $L$. The
variety of metabelian Lie algebras over $F$ is denoted by $A^{2}$. A
set $S \subset L(X)$ of ordinary polynomials (respectively graded
polynomials) is called a basis for the ordinary identities
(respectively graded identities) of a Lie algebra (respectively a
graded Lie algebra) $A$ when $Id(A) = \langle S \rangle_{T}$
(respectively $Id_{G}(A) = \langle S \rangle_{T}$).

\begin{example}
In 1990's, Semenov (Proposition 1, \cite{Semenov}) proved that
\begin{flushleft}
$Sem_{1}(x_{1},x_{2}) = (x_{1})f(ad(x_{2})), \ \ f(t) = t^{q^{2} +
2} - t^{3}$,
\end{flushleft}
\begin{flushleft}
$Sem_{2}(x_{1},x_{2}) = [x_{1},x_{2}] - [x_{1},x_{2},(x_{1}
)^{q^{2}-1}] - [x_{1},(x_{2})^{q}] +
[x_{1},x_{2},(x_{1})^{q^{2}-1},(x_{2})^{q-1}] + [x_{1},x_{2}
,((x_{1})^{q^{2}} - (x_{1})),[x_{1},x_{2}]^{q-2},(x_{2})^{q^{2}} -
(x_{2})]- [x_{2},([(x_{1})^{q^{2}} - (x_{1}),x_{2}
])^{q},((x_{2})^{q^{2} - 2} - (x_{2})^{q - 2})]$.
\end{flushleft}
are polynomial identities of $sl_{2}(F)$.
\end{example}

A finite-dimensional ordinary (respectively graded) Lie algebra $L$
is critical if $var(L)$ (respectively $var_{G}(L)$) is not generated
by all proper subquotients of $L$. It is monolithic if it contains a
single ordinary (respectively graded) minimal ideal. This single
ideal is termed a monolith. It is known that if $L$ is an ordinary
(respectively graded) critical Lie algebra, then $L$ is monolithic.
Notice that if $L = \oplus_{g \in G} L_{g}$ is a critical ordinary
Lie algebra, then $L$ is critical as a $G$-graded Lie algebra.

\begin{example}
If $L$ is a critical abelian (respectively graded) Lie algebra, then
$dim L = 1$. If $L$ is a two-dimensional (non abelian) metabelian
Lie algebra, then $L$ is critical. Furthermore, $sl_{2}(F)$ is a
critical Lie algebra.
\end{example}

\begin{proposition}\label{tow}
Let $L = \bigoplus_{g \in G}L_{g}$ be a finite-dimensional (non
abelian) metabelian graded Lie $A$-algebra over an arbitrary field
$F$. If $L$ is monolithic, then $Nil(L) = [L,L]$.
\end{proposition}
\begin{proof}
According to Theorem \ref{david}, $Nil(L) = [L,L] +^{\cdot} Z(L)$.
By hypothesis, $L$ is monolithic. Thus, $Z(L) = \{0\}$ and $Nil(L) =
[L,L]$.
\end{proof}


The next theorem describes the relationship between critical
metabelian Lie $A$-algebras and monolithic Lie $A$-algebras.

\begin{theorem}[Sheina, Theorem 1,\cite{Sheina}]\label{sheina}
A finite-dimensional monolithic Lie $A$-algebra $L$ over an
arbitrary finite field is critical if, and only if, its derived
algebra cannot be represented as a sum of two ideals strictly
contained within it.
\end{theorem}


A locally finite Lie algebra is a Lie algebra for which every
finitely generated subalgebra is finite. A variety of Lie algebras
(respectively graded Lie algebras) is said to be locally finite when
every finitely generated Lie algebra (respectively graded Lie
algebra) has finite cardinality. It is known that a variety
generated by a finite Lie algebra (respectively a graded finite Lie
algebra) is locally finite. As in the ordinary case, if a variety of
graded Lie algebras is locally finite, then it is generated by its
critical algebras. For more details about varieties of Lie algebras,
see Chapters 4 and 7 of \cite{Bahturin}.

The next result will prove useful for our purposes.

\begin{theorem}[Semenov, Proposition 2,
\cite{Semenov}]\label{chan} Let $\mathcal{B}$ be a variety of
ordinary Lie algebras over a finite field $F$. If there exists a
polynomial $f(t) = a_{1}t + \ldots + a_{n}t^{n} \in F[t]$ such that
$yf(adx) := a_{1}[y,x] + \ldots + a_{n}[y,x^{n}] \in
Id(\mathcal{B})$, then $\mathcal{B}$ is a locally finite variety.
\end{theorem}


Let $L_{1}$ and $L_{2}$ be two graded Lie algebras (finite
-dimensional), and $I_{1} \subset L_{1}$ and $I_{2} \subset L_{2}$
be graded ideals. We say that $I_{1}$ (in $L_{1}$) is similar to
$I_{2}$ (in $L_{2}$) ($I_{1} \trianglelefteq A_{1} \sim I_{2}
\trianglelefteq A_{2}$) if there exist isomorphisms $\alpha_{1} :
I_{1} \rightarrow I_{2}$ and $\alpha_{2}:
\frac{L_{1}}{C_{L_{1}}(I_{1})} \rightarrow
\frac{L_{2}}{C_{L_{2}}(I_{2})}$ such that for all $a \in I_{1}$ and
$b + C_{L_{1}}(I_{1}) \in \frac{L_{1}}{C_{L_{1}}(I_{1})}$:
\begin{center}
$\alpha_{1}([a,c]) = [\alpha_{1}(a),d]$,
\end{center}
where $c + C_{L_{1}}(I_{1}) = b + C_{L_{1}}(I_{1}) \ \mbox{and} \ d
+ C_{L_{2}}(I_{2}) = \alpha_{2}(b + C_{L_{1}}(I_{1})) $.

By proceeding as in \cite{Hanna} (cf. pages 162 to 166), we have the
following.

\begin{proposition}\label{hanna-teo}
If two critical graded Lie algebras $L_{1}$ and $L_{2}$ generate the
same variety, then their monoliths are similar.
\end{proposition}

\section{$\mathbb{Z}_{2}$-graded identities of $sl_{2}(F)$}

From now on, we denote by $Y = \{y_{1},\ldots,y_{n},\ldots\}$ the
even variables, by $Z = \{z_{1},\ldots,z_{n},\ldots\}$ the odd
variables.

\begin{lemma}\label{lema8}
Let $sl_{2}(F)$ be the Lie algebra of traceless $2\times 2$ matrices
over $F$ endowed with the natural grading. The following polynomials
are graded identities of $sl_{2}(F)$
\begin{center}
$[y_{1},y_{2}],[z_{1},y_{1}^{q}] = [z_{1},y_{1}]$.
\end{center}
\end{lemma}
\begin{proof}
It is clear that $[y_{1},y_{2}] \in Id_{G}(sl_{2}(F))$, because the
diagonal is commutative. Choose $a_{i} = \lambda_{11,i}e_{11} -
\lambda_{11,i}e_{22}$ and $b_{j} = \lambda_{12,j}e_{12} +
\lambda_{21,j}e_{21}$, so:
\begin{center}
$[b_{j},a_{i}^{q}] = \lambda_{11,i}^{q}[b_{j}, {h}^{q}] =
\lambda_{11,i}^{q}((-2)^{q}\lambda_{12,j}e_{12} +
2^{q}\lambda_{21,j}e_{21}) = \lambda_{11,i}(-2\lambda_{12,j}e_{12} +
2\lambda_{21,j}e_{21}) = [b_{j},a_{i}]$.
\end{center}
Thus, $[z_{1},y_{1}^{q}] = [z_{1},y_{1}] \in Id_{G}(sl_{2}(F))$. The
proof is complete.
\end{proof}

We now cite two papers. First we present a corollary of Bahturin et
al. in \cite{Bahturin2}.

\begin{proposition}[Bahturin et al., Corollary 1,
\cite{Bahturin2}]\label{bah} Let $R = M_{n}(F)$, $char F = p > 0$,
$p \neq 2$. Let $G$ be an elementary abelian $p$-group. Suppose that
$R = \bigoplus_{g \in G} R_{g}$ is a grading on $R^{(-)}$. Then $R =
\bigoplus_{g \in G} R_{g}$ is a grading on $R$ if and only if $1 \in
R_{e}$.
\end{proposition}
\begin{Remark}
Here, $1$ denotes the identity matrix of $M_{n}(F)$, and $e$ denotes
the identity element of $G$.
\end{Remark}

In this study, $F$ is a finite field of $char F = p > 3$ and size
$q$. So, there exists $b \in F - \{0\}$ that is not a perfect
square.

Notice that if $sl_{2}(F) = (sl_{2}(F))_{0} \oplus (sl_{2}(F))_{1}$
is a $\mathbb{Z}_{2}$-grading on $sl_{2}(F)$, then $((sl_{2}(F))_{0}
\oplus F(e_{11} + e_{22})) \oplus (sl_{2}(F))_{1}$ is a
$\mathbb{Z}_{2}$-grading on $sl_{2}(F)$. By Proposition \ref{bah},
$((sl_{2}(F))_{0} \oplus F(e_{11} + e_{22})) \oplus (sl_{2}(F))_{1}$
is a $\mathbb{Z}_{2}$-grading on $M_{2}(F)$. The next proposition
describes the $\mathbb{Z}_{2}$-grading on $M_{2}(F)$.

\begin{proposition}[Khazal et al., Theorem 1.1, adapted, \cite{Khazal}]\label{kah}
Let $F$ be a field of $char F \neq 2$. Then any
$\mathbb{Z}_{2}$-grading of $M_{2}(F)$ is isomorphic to one of the
following.
\begin{description}
\item $(M_{2}(F)_{0},M_{2}(F)_{1}) = (M_{2}(F),0)$;
\item $(M_{2}(F)_{0},M_{2}(F)_{1}) = (Fe_{11} \oplus Fe_{22},Fe_{12} \oplus
Fe_{21})$;
\item $(M_{2}(F)_{0},M_{2}(F)_{1}) =$
\subitem $(F(e_{11} + e_{22})\oplus F(e_{12} + be_{21}),F(e_{11} -
e_{22})\oplus F(e_{12} - be_{21}))$, where $b \in F - F^{2}$.
\end{description}
\end{proposition}
\begin{Remark}\label{boboc}
It is well known that $(F - \{0\},.)$ is a cyclic group of order
$q-1$. By elementary theory of groups, for every divisor $d$ of $q -
1$, there exists a unique subgroup $H'$ of $(F - \{0\},.)$ of order
$d$. Let $H$ be the subgroup of order $\frac{q - 1}{2}$. It is easy
to see that there exists $b' \in (F - F^{2})\cap (F - H)$. Finally,
note that
\begin{center}
$[(e_{11} - e_{22}),(e_{12} + b'e_{21})] \neq [(e_{11} -
e_{22}),(e_{12} + b'e_{21}),\ldots,(e_{12} + b'e_{21})]$,
\end{center}
where $(e_{12} + b'e_{21})$ appears $q$ times in the expanded
commutator.
\end{Remark}

\begin{proposition}\label{Final}
Let $sl_{2}(F) = (sl_{2}(F))_{0} \oplus (sl_{2}(F))_{1}$ be a
$\mathbb{Z}_{2}$-grading on $sl_{2}(F)$ having the following
characteristics:
\begin{description}
\item $dim \ (sl_{2}(F))_{0} = 1$,
\item $[a,c^{q}] = [a,c]$ for all $a \in (sl_{2}(F))_{1}$ and $c \in
F(e_{11} + e_{22})\oplus (sl_{2}(F))_{0}$.
\end{description}
Then the $\mathbb{Z}_{2}$-gradings $((sl_{2}(F))_{0},
(sl_{2}(F))_{1})$ and $(F(e_{11} - e_{22}), Fe_{12} \oplus Fe_{21})$
are isomorphic.
\end{proposition}
\begin{proof}
First, note that $(((sl_{2}(F))_{0} \oplus F(e_{11} + e_{22}))
\oplus (sl_{2}(F))_{1})$ is a $\mathbb{Z}_{2}$-grading on
$gl_{2}(F)$. According to Proposition \ref{bah}, $(((sl_{2}(F))_{0}
\oplus F(e_{11} + e_{22})) \oplus (sl_{2}(F))_{1})$ is a
$\mathbb{Z}_{2}$-grading on $M_{2}(F)$. It is clear that this
grading on $M_{2}(F)$ is not isomorphic to the first grading
presented in Proposition \ref{kah}. Notice also that $(F(e_{11} +
e_{22})\oplus (sl_{2}(F))_{0},(sl_{2}(F))_{1})$ cannot be isomorphic
to the third grading presented in Proposition \ref{kah}, because
$[z_{1},y_{1}] = [z_{1},y_{1}^{q}]$ is not a polynomial identity for
$M_{2}(F)$ endowed with third grading (Remark \ref{boboc}).
According to Proposition \ref{kah}, there exists a $G$-graded
isomorphism $\phi: M_{2}(F) \rightarrow M_{2}(F)$ such that
\begin{center}
$\phi((sl_{2}(F))_{0} \oplus F(e_{11} + e_{22})) = Fe_{11} \oplus
Fe_{22}$ and $\phi((sl_{2}(F))_{1}) = Fe_{12} \oplus Fe_{21}$.
\end{center}
Note that $\phi: gl_{2}(F) \rightarrow gl_{2}(F)$ is an isomorphism
of Lie algebras and $\phi(sl_{2}(F)) = sl_{2}(F)$. Thus,
$((sl_{2}(F))_{0}, (sl_{2}(F))_{1})$ and $(F(e_{11} - e_{22}),
Fe_{12} \oplus Fe_{21})$ are isomorphic. The proof is complete.
\end{proof}

Henceforth we consider only $sl_{2}(F)$ and $Fe_{11} \oplus Fe_{12}$
endowed with the natural grading by $(\mathbb{Z}_{2},+)$. Recall
that $Sem_{1}(x_{1},x_{2}), Sem_{2}(x_{1},x_{2}) \in Id(sl_{2}(F))$.
We denote by $S$ the set with following polynomials.

\begin{flushleft}
$Sem_{1}(y_{1}+z_{1},y_{2}+z_{2}),Sem_{2}(y_{1}+z_{1},y_{2}+z_{2}),[y_{1},y_{2}],
\mbox{and} \ [z_{1},y_{1}^{q}] = [z_{1},y_{1}]$.
\end{flushleft}

\begin{corollary}\label{localmente}
The variety $\mathcal{V}(S)$ is locally finite.
\end{corollary}
\begin{proof}
Let $L = L_{0} \oplus L_{1} \in \mathcal{V}(S)$ be a finitely
generated algebra. By definition of $S$, $Sem_{1}(y_{1} +
z_{1},y_{2} + z_{2}) \in Id_{G}(L)$. Hence, $Sem_{1}(x_{1},x_{2})
\in Id(L)$. So, by Theorem \ref{chan}, it follows that $L$ is a
finite Lie algebra.
\end{proof}

\begin{corollary}\label{ALie}
Let $L \in \mathcal{V}(S)$ be a finite-dimensional Lie algebra. Then
every nilpotent subalgebra of $L$ is abelian.
\end{corollary}
\begin{proof}
From the definition of $S$, it follows that $Sem_{2}(x_{1},x_{2})
\in Id(L)$. Let $M \neq \{0\}$ be a nilpotent (unnecessarily graded)
subalgebra of $L$. If $M^{t} = \{0\}$ for a positive integer $t \leq
q + 1$, it is clear that $M$ is abelian. If the index of nilpotency
is equal to $q + 2$, then $\frac{M}{Z(M)}$ is abelian. Consequently,
$M$ is abelian. Induction on the index of nilpotency will give the
desired result.
\end{proof}

It is well known that a verbal ideal (and, respectively, a graded
verbal ideal) over an infinite field is multi homogeneous. This fact
can be weakened, as stated in the next lemma.

\begin{lemma}
Let $I$ be a graded verbal ideal over a field of size $q$. If
$f(x_{1},\ldots,x_{n}) \in I$ and $0 \leq deg_{x_{1}}
f,\ldots,deg_{x_{n}} f < q$, then each multi homogeneous component
of $f$ belongs to $I$ as well.
\end{lemma}

\begin{lemma}\label{lema5}
If $L = span_{F}\{e_{11}, e_{12}\} \subset gl_{2}(F)$, then the
$\mathbb{Z}_{2}$-graded identities of $L$ follow from
\begin{center}
$[y_{1},y_{2}], [z_{1},z_{2}]$ and $[z_{1},y_{1}^{q}] =
[z_{1},y_{1}]$.
\end{center}
\end{lemma}
\begin{proof}
It is clear that $L$ satisfies the identities $[y_{1},y_{2}],
[z_{1},z_{2}]$ and $[z_{1},y_{1}^{q}] = [z_{1},y_{1}]$. We will
prove that the reverse inclusion holds true. Let $f$ be a polynomial
identity of $L$. We may write
\begin{center}
$f = g + h$,
\end{center}
where $h \in \langle [y_{1},y_{2}], [z_{1},z_{2}], [z_{1},y_{1}^{q}]
= [z_{1},y_{1}] \rangle$ and $g(x_{1},\ldots,x_{n}) \in Id_{G}(L)$,
with $0 \leq deg_{x_{1}} g,\ldots,deg_{x_{n}} g < q$. In this way,
we may suppose that $g$ is a multi homogeneous polynomial. If
$g(y_{1}) = \alpha_{1}.y_{1}$ or $g(z_{1}) = \alpha_{2}z_{1}$ we can
easily see that $\alpha_{1} = \alpha_{2} = 0$. In the other case, we
may assume that
\begin{center}
$g(z_{1},y_{1},\ldots,y_{l}) =
\alpha_{3}.[z_{1},y_{1}^{a_{1}},\ldots,y_{l}^{a_{l}}], 1 \leq
a_{1},\ldots,a_{l} < q$.
\end{center}
However, $g(e_{12},e_{11},\ldots,e_{11})$ is a non zero multiple
scalar of $e_{12}$, and consequently, $\alpha_{3} = 0$. Hence, $f =
h$ and the proof is complete.
\end{proof}

\begin{lemma}\label{lema6}
Let $L = L_{0} \oplus L_{1} \in A^{2}\cap \mathcal{V}(S)$ be a
critical Lie $A$-algebra. Then $\newline L \in
var_{G}(span_{F}\{e_{11},e_{12}\})$.
\end{lemma}
\begin{proof}
According to Lemma \ref{lema5}, it is sufficient to prove that $L$
satisfies the identity $[z_{1},z_{2}]$.

By assumption, $L$ is critical and therefore $L$ is monolithic. If
$L$ is abelian, then $dim L = 1$. In this case $L \cong
span_{F}\{e_{11}\}$ or $L \cong span_{F}\{e_{12}\}$.

In the sequel, we suppose that $L$ is non abelian . From Proposition
\ref{tow}, we have $[L,L] = Nil(L) = [L_{1},L_{1}] \oplus
[L_{0},L_{1}]$. From the identity $[z_{1},y_{1}] =
[z_{1},y_{1}^{q}]$, $\{0\} = [L_{1},[L_{1},L_{1}]] =
-[L_{1},L_{1},L_{1}]$. So, by the identity $Sem_{2}(y_{1} +
z_{1},y_{2} + z_{2})$, we have $[z_{1},z_{2}] \in
Id_{G}(span_{F}\{e_{11},e_{12}\})$, as required. The proof is
complete.
\end{proof}

\begin{corollary}\label{corolario1}
$A^{2}\cap var_{G}(sl_{2}(F)),A^{2}\cap \mathcal{V}(S) \mbox{and} \
\ var_{G}(span_{F}\{e_{11},e_{12}\})$ coincide.
\end{corollary}
\begin{proof}
First, notice that $A^{2}\cap var_{G}(sl_{2}(F)) \subset A^{2}\cap
\mathcal{V}(S)$ which is a locally finite variety. By Lemma
\ref{lema6}, all critical algebras of $A^{2}\cap \mathcal{V}(S)$
belong to $var_{G}(span_{F}\{e_{11},e_{12}\}) \newline \subset
A^{2}\cap var_{G}(sl_{2}(F))$. Therefore, $A^{2}\cap \mathcal{V}(S)
\subset var_{G}(span_{F}\{e_{11},e_{12}\})$. Thus, we have
\begin{center}
$A^{2} \cap var_{G}(sl_{2}(F)) = A^{2} \cap \mathcal{V}(S) =
var_{G}(span_{F}\{e_{11},e_{12}\})$.
\end{center}
\end{proof}

\begin{lemma}\label{soluvel}
Let $L$ be a critical solvable Lie $A$-algebra belonging to
$\mathcal{V}(S)$. Then $L$ is metabelian.
\end{lemma}
\begin{proof}
Let $L$ be a critical (non abelian) solvable Lie algebra that
belongs to $\mathcal{V}(S)$ with monolith $W$. By Proposition
\ref{premet}, we have $[L,L]\cap Z(L) = \{0\}$. Consequently, $Z(L)
= \{0\}$. Notice that $Z(C_{L}(Nil(L))) = Nil(L)$. If $(Nil(L))_{1}
= L_{1}$, then $L$ is metabelian. Now, we suppose that $(Nil(L))_{1}
\varsubsetneq L_{1}$. We assert that $(Nil(L))_{0} = \{0\}$.
Suppose, on the contrary, that there exists  $a \neq 0 \in
(Nil(L))_{0}$. Hence, there exists $b \in L_{1} - (Nil(L))_{1}$ such
that $[b,a] \neq 0$, because $Z(L) = \{0\}$. However, $[b,a] =
[b,a^{q}] = 0$. This is a contradiction. Thus, $[L_{1},Nil(L)] =
\{0\}$. Consequently $C_{L}(Nil(L)) \supset L_{1} \cup
[L_{1},L_{1}]$. By Proposition \ref{premet}
\begin{center}
$Z(C_{L}(Nil(L))) \cap [C_{L}(Nil(L)),C_{L}(Nil(L))] = \{0\}$.
\end{center}
Hence, $[C_{L}(Nil(L)),C_{L}(Nil(L))] = \{0\}$. So, $L^{(2)} =
\{0\}$ and the proof is complete.
\end{proof}

\begin{lemma}\label{corolario2}
Let $L$ be a critical non-solvable Lie $A$-algebra belonging to
$\mathcal{V}(S)$. Then $L$ is $G$-simple.
\end{lemma}
\begin{proof}
Let $W$ be the monolith of $L$. We claim that $L$ is semisimple.
Suppose on the contrary that $Rad(L) \neq \{0\}$. Thus, $W \subset
Rad(L) \cap [L,L]$. The non trivial subspace $W$ is an abelian ideal
and it is contained in $L^{(n)}$, where $n$ is the least nonnegative
integer such that $L^{(n)} = L^{(n+1)}$. According to Proposition
\ref{premet}, $[L,L] \cap Z(L) = \{0\}$. So $Z(L) = \{0\}$ and
$[L,W] = W$. The identities $[y_{1},y_{2}]$ and $[z_{1},y_{1}] =
[z_{1},y_{1}^{q}]$ mean that the subspace $W_{0} = \{0\}$. Notice
that $[W,[L,L]] = \{0\} = [W,L^{(n)}]$ and $Z(L^{(n)}) \supset W$.
By Proposition \ref{premet}, $Z(L^{(n)}) \cap L^{(n)} = \{0\}$. This
is a contradiction, so $L$ is semisimple. By Propositions
\ref{premet} and \ref{premet3}, $L$ is a direct sum of $G$-simple
Lie algebras. Given that $L$ is monolithic, we conclude that $L$ is
$G$-simple.
\end{proof}

The next theorem was proved by Drensky in (\cite{Drensky} Lemma,
page 991).

\begin{lemma}\label{diagonal}
Let $V$ be a finite dimensional vector space over $F$ and let $A$ be
an abelian Lie algebra of the linear transformations $\phi: V
\rightarrow V$, where each has the equality
\begin{center}
$\phi^{q} = \phi$.
\end{center}
Then, every $\phi \in A$ is diagonalizable.
\end{lemma}

\begin{definition}
Let $L$ be a finite dimensional Lie algebra with a diagonalizable
operator $T: L \rightarrow L$. We denote by $V(T)$ a basis of $L$
formed by the eigenvectors of $T$. Moreover, we denote
$V(T)_{\lambda} = \{v \in V(T)| T(v) = \lambda.v\}$. We denote
$EV(w)$ the eigenvalue associated with the eigenvector $w \in V(T)$.
\end{definition}

Let $L \in \mathcal{V}(S)$ be a finite dimensional Lie algebra. It
is not difficult to see that $ad(L_{0}) = \{ad a: L \rightarrow L |
a \in L_{0}\}$ is an abelian subalgebra of linear transformations of
$L$. Moreover, $(ad a_{0})^{p} = ad a_{0}$ for all $a_{0} \in
L_{0}$. By Lemma \ref{diagonal}, we have the following.

\begin{corollary}\label{corolario6}
Let $L \in \mathcal{V}(S)$ be a finite dimensional Lie algebra. Let
$a_{0} \in L_{0}$. Then there exists $V(ad a_{0}) \subset L_{0} \cup
L_{1}$.
\end{corollary}


\begin{proposition}\label{remark}
Let $L \in \mathcal{V}(S)$ be a finite dimensional $G$-simple
algebra. Let $a_{0} \in L_{0}$. Then there exists $V(ad a_{0})
\subset L_{0} \cup L_{1}$. Moreover, $V(ad a_{0})_{0} \cap L_{1} =
\emptyset$ for any basis $V(ad a_{0}) \subset L_{0} \cup L_{1}$.
\end{proposition}
\begin{proof}

According to Corollary \ref{corolario6}, there exists $V(ad a_{0})
\subset L_{0} \cup L_{1}$.

Let $b_{1},b_{2} \in V(ad a_{0})\cap L_{1}$. Notice that if
$[b_{1},b_{2}] \neq 0$, then $EV(b_{1}) = - EV(b_{2})$.

It is clear that $\langle a_{0} \rangle$ is a graded ideal, and that
it is equal to $L$. Notice also that $L =
span_{F}\{[a_{0},b_{1},\ldots,b_{n}]| b_{1},\ldots, b_{n} \in V(ad
a_{0}), n \geq 1\}$.

If there was a non zero element $b \in V(ad a_{0})_{0} \cap L_{1}$,
we could easily check that $[a_{0},b_{1},b] = 0$ for any $b_{1} \in
V(ad (a_{0}))$. More generally, by an inductive argument and routine
calculations, we would have $[a_{0},b_{1},\ldots,b_{n},b] = 0$ for
any $n \geq 1$ and $b_{1},\ldots,b_{n} \in V(ad a_{0})$. However, an
element such as $b$ cannot be, because $Z(L) = \{0\}$. So, $V(ad
a_{0})_{0} \cap L_{1} = \emptyset$.
\end{proof}


\begin{lemma}\label{semenov}
Let $L \in \mathcal{V}(S)$ be a critical non solvable algebra, then
$L \cong sl_{2}(F)$.
\end{lemma}
\begin{proof}

First of all, notice that $dim L_{0} \geq 1$ and $dim L_{1} \geq 2$.
According to Lemma \ref{corolario2} $L$ is $G$-simple. Let $a_{0}
\in L_{0}$. By Proposition \ref{remark}, there exists $V(ad a_{0}) =
\{b_{1},\ldots,b_{n}\} \subset L_{0}\cup L_{1}$. Moreover, $V(ad
a_{0})_{0} \cap L_{1} = \emptyset$.

Let $-\lambda_{1} \leq \ldots \leq - \lambda_{m} < 0 < \lambda_{m}
\leq \ldots \leq \lambda_{1}$ be the eigenvalues associated with the
eigenvectors of $V(ad a_{0})$. Notice that
\begin{center}
$L_{0} = \sum\limits_{i=1}^{m} span_{F}\{[V(ad a_{0})_{\lambda_{i}},
V(ad a_{0})_{-\lambda_{i}}]\}$.
\end{center}
Without loss of generality, suppose that $span_{F}\{[V(ad
a_{0})_{\lambda_{1}},V(ad a_{0})_{-\lambda_{1}}]\} \neq \{0\}$. We
assert that $span_{F}\{[V(ad a_{0})_{\lambda_{1}},V(ad
a_{0})_{-\lambda_{1}}]\} \oplus span_{F}\{V(ad
a_{0})_{\lambda_{1}}\}$ is a subalgebra of $L$.

In fact, let $a \in V(ad a_{0})_{\lambda_{1}}$ and $b \in
span_{F}\{[V(ad a_{0})_{\lambda_{1}},V(ad a_{0})_{-\lambda_{1}}]\}$.
Consider $[a,b] = \sum\limits_{i=1}^{n} \alpha_{i}b_{i}$. So,
\begin{center}
$[a,b,a_{0}] = -\sum\limits_{i=1}^{n}\alpha_{i}.EV(b_{i})b_{i}$.
\end{center}
On the other hand,
\begin{center}
$[a,b,a_{0}] = -\lambda_{1}[a,b] =
-\lambda_{1}(\sum\limits_{i=1}^{n}\alpha_{i}b_{i})$.
\end{center}
Hence
\begin{center}
$(-EV(b_{j}).\alpha_{j} + \lambda_{1}.\alpha_{j})b_{j} = 0$.
\end{center}
Consequently, if $\alpha_{j} \neq 0$, then $\lambda_{1} =
EV(b_{j})$.

Similarly, the subspace $span_{F}\{[V(ad a_{0})_{\lambda_{1}},V(ad
a_{0})_{-\lambda_{1}}]\} \oplus span_{F}\{V(ad
a_{0})_{-\lambda_{1}}\}$ is a subalgebra. Notice that
\begin{center}
$span_{F}\{[V(ad a_{0})_{\lambda_{1}},V(ad
a_{0})_{-\lambda_{1}}]\}\oplus span_{F}\{V(ad a_{0})_{\lambda_{1}}\}
\oplus span_{F}\{V(ad a_{0})_{-\lambda_{1}}\}$
\end{center}
is a graded ideal of $L$.

Therefore, $L_{0} = span_{F}\{V(ad a_{0})_{0}\} = span_{F}\{[V(ad
a_{0})_{\lambda_{1}},V(ad a_{0})_{-\lambda_{1}}]\}$ and the subspace
$L_{1}$ is equal to $span_{F}\{V(ad a_{0})_{\lambda_{1}}\}\oplus
span_{F}\{V(ad a_{0})_{-\lambda_{1}}\}$.


Notice that $span_{F}\{V(ad a_{0})_{\lambda_{1}}\}$ is an
irreducible $L_{0}$-module, because $L$ is $G$-simple. Moreover, it
is not difficult to see that $L_{0} \oplus span_{F}\{V(ad
a_{0})_{\lambda_{1}}\}$ is a monolithic metabelian Lie algebra with
monolith $span_{F}\{V(ad a_{0})_{\lambda_{1}}\}$ when viewed as
ordinary Lie algebra. Notice that
\begin{center}
$[L_{0} \oplus span_{F}\{V(ad a_{0})_{\lambda_{1}}\}, L_{0} \oplus
span_{F}\{V(ad a_{0})_{\lambda_{1}}\}] = span_{F}\{V(ad
a_{0})_{\lambda_{1}}\}$
\end{center}
cannot be represented by the sum of two ideals strictly contained
within it. By Theorem \ref{sheina}, $L_{0} \oplus span_{F}\{V(ad
a_{0})_{\lambda_{1}}\}$ is critical when viewed as an ordinary Lie
algebra. Thus, it is critical when viewed as a graded algebra as
well.

Following the arguments of Lemma \ref{lema5}, we can prove that
\begin{center}
$Id_{G}(L_{0} \oplus span_{F}\{V(ad a_{0})_{\lambda_{1}}\}) =
\langle [y_{1},y_{2}], [z_{1},y_{1}] = [z_{1},y_{1}^{q}],
[z_{1},z_{2}]\rangle_{T}$.
\end{center}
Consequently, it follows from Proposition \ref{hanna-teo} that
$span_{F}\{V(ad a_{0})_{\lambda_{1}}\}$ is a one-dimensional vector
space. Analogously, we have $dim (span_{F}\{V(ad
a_{0})_{-\lambda_{1}}\}) = 1$. Therefore, $L_{0} \oplus
span_{F}\{V(ad a_{0})_{\lambda_{1}}\} \oplus span_{F}\{V(ad
a_{0})_{-\lambda_{1}}\}$ is a three-dimensional $G$-simple Lie
algebra. So, $L$ is simple and isomorphic to $sl_{2}(F)$ (as
ordinary Lie algebras). Hence, by Proposition \ref{Final}, $L \cong
sl_{2}(F)$ (as graded Lie algebras), where $sl_{2}(F)$ is naturally
graded by $\mathbb{Z}_{2}$. The proof is complete.
\end{proof}

\section{Main theorem}

We now prove the main theorem of this paper.

\begin{theorem}
Let $F$ be a field of $char(F) > 3$ and size $|F| = q$. The
$\mathbb{Z}_{2}$-graded identities of $sl_{2}(F)$ follow from
\begin{center}
$[y_{1},y_{2}], Sem_{1}(y_{1}+z_{1},y_{2}+z_{2}),
Sem_{2}(y_{1}+z_{1},y_{2}+z_{2}), \mbox{and} \ [z_{1},y_{1}] =
[z_{1},y_{1}^{q}]$.
\end{center}
\end{theorem}
\begin{proof}
It is clear that $var_{G}(sl_{2}(F)) \subset \mathcal{V}(S)$. To
prove that the reverse inclusion holds, it is sufficient to prove
that all critical algebras of $\mathcal{V}(S)$ are also critical
algebras of $var_{G}(sl_{2}(F))$. According to Corollary
\ref{corolario1}, $A^{2}\cap \mathcal{V}(S) = A^{2}\cap
var_{G}(sl_{2}(F))$. By Lemma \ref{soluvel}, any critical solvable
Lie algebra of $\mathcal{V}(S)$ is metabelian. By Lemma
\ref{semenov}, any critical non solvable Lie algebra of
$\mathcal{V}(S)$ is isomorphic to $sl_{2}(F)$. Therefore,
$\mathcal{V}(S) \subset var_{G}(sl_{2}(F))$, and the theorem is
proved.
\end{proof}

\section{Acknowledgments}

The author thanks CNPq for his Ph.D. scholarship. Moreover, the
author thanks the reviewer for his/her comments.

\end{document}